\RequirePackage{etex}
\RequirePackage{easymat}

\documentclass[12pt, a4paper]{elsarticle}
\usepackage{amsthm,amsmath,amssymb,extsizes,easymat,graphicx,hyperref,easybmat}

\usepackage[active]{srcltx}
\usepackage{MnSymbol}
\usepackage{mathrsfs}
\DeclareMathOperator{\cod}{cod}
\DeclareMathOperator{\diag}{diag}
\DeclareMathOperator{\orb}{O}

\DeclareMathOperator{\tsp}{T}

\DeclareMathOperator{\tr}{trace}

\DeclareMathOperator{\GSYL}{GSYL}

\DeclareMathOperator{\rev}{rev}
\DeclareMathOperator{\POL}{POL}
\DeclareMathOperator{\PEN}{PENCIL}
\DeclareMathOperator{\skewPEN}{SSPENCIL}

\newcommand{\sdotsss}%
{\text{\raisebox{-2.2pt}{$\cdot\,$}%
 \raisebox{1.7pt}{$\cdot$}%
\raisebox{5.6pt}{$\,\cdot$}}}

\renewcommand{\le}{\leqslant}
\renewcommand{\ge}{\geqslant}

\newtheorem{theorem}{Theorem}[section]

\newtheorem{lemma}[theorem]{Lemma}

\newtheorem{definition}[theorem]{Definition}

\newtheorem{remark}[theorem]{Remark}

\newcommand{\hide}[1]{}

\begin{document}

\title{Generic skew-symmetric matrix polynomials with fixed rank and fixed odd grade\tnoteref{t1}}

\tnotetext[t1]{Preprint Report UMINF 17.07, Department of Computing Science, Ume\r{a} University}

\date{}

\author[um]{Andrii Dmytryshyn}
\ead{andrii@cs.umu.se}

\author[cm]{Froil\'an M. Dopico}
\ead{dopico@math.uc3m.es}

\address[um]{Department of Computing Science, Ume\aa \, University, SE-901 87 Ume\aa, Sweden.}

\address[cm]{Departamento de Matem\'aticas, Universidad Carlos III de Madrid, Avenida de la Universidad 30, 28911, Legan\'es, Spain.}

\begin{abstract}
We show that the set of $m \times m$ complex skew-symmetric matrix polynomials of odd grade $d$, i.e., of degree at most $d$, and (normal) rank at most $2r$ is the closure of the single set of matrix polynomials with the certain, explicitly described, complete eigenstructure. This complete eigenstructure corresponds to the most generic $m \times m$ complex skew-symmetric matrix polynomials of odd grade $d$ and rank at most $2r$. In particular, this result includes the case of skew-symmetric matrix pencils ($d=1$).
\end{abstract}


\maketitle

\section{Introduction}

The structure of the sets of matrix pencils and of matrix polynomials with fixed grade and fixed rank is not trivial and, as a consequence, only recently has been investigated in the literature \cite{DeDo08,DeDoLa17,DmDo16}, with the main purpose of providing reasonably simple descriptions of these sets as the closures of certain ``generic sets'' of pencils and matrix polynomials that can be easily described in terms of their eigenstructures or in terms of certain parameterizations. One motivation for this type of research comes from its applications in the study of the effect of low rank perturbations on the spectral information of pencils and matrix polynomials, since these problems have received considerable attention in the last years. See for instance the references  \cite{Batzke14,Batzke15,Batzke16,DeDo07,DeDo16,DeDoMo08,MMW17} for different low rank perturbation problems related to matrix pencils, and the particular study in \cite{DeDo09} on certain low rank perturbations of matrix polynomials. Moreover, low rank perturbations of matrix pencils have been applied recently to some classical problems as the eigenvalue placement problem \cite{GerTrunk17} or the estimation of the distance of a regular pencil to the nearest singular pencil \cite{MMW15}. Finally, from a more theoretical perspective, the study of the sets of matrix pencils and of matrix polynomials with fixed grade and fixed rank generalizes  classical studies \cite{Wate84} on the algebraic structure of the set of $n\times n$ singular pencils, i.e., those whose rank is at most $n-1$.

As can be observed in the titles of several of the references mentioned in the paragraph above, some of the problems concerning low rank perturbations of matrix pencils involve pencils with particular structures (symmetric, Hermitian, palindromic, skew-symmetric, alternating, etc), since the pencils and matrix polynomials arising in applications have often such particular structures \cite{MMMM06b}. In this scenario, it is natural to consider the extension of the results in \cite{DeDo08,DeDoLa17,DmDo16} to the structured setting, with the aim of providing simple descriptions of sets of structured matrix pencils and of structured matrix polynomials with fixed grade and fixed rank in the case of the structures appearing in applications. The results in this paper can be seen as the first step available in the literature towards the solution of this ambitious and wide problem. In particular, we prove that the set of $m \times m$ complex skew-symmetric matrix polynomials of odd grade $d$, i.e., of degree at most $d$, and (normal) rank at most $2r$ is the closure of the single set of matrix polynomials with the certain, explicitly described, complete eigenstructure, which is termed as the {\em unique generic eigenstructure} of this set of skew-symmetric matrix polynomials of odd grade $d$ and rank at most $2r$. By taking $d=1$, we obtain that such uniqueness also holds for skew-symmetric pencils. This result greatly simplifies the description of the set of skew-symmetric polynomials of fixed odd grade and fixed rank and illustrates how strong can be the effect of imposing  a structure on this type of problems, since the corresponding unstructured problem  has many more generic complete eigenstructures. More precisely, it has been proved in  \cite{DeDo08,DmDo16} that there are $rd+1$ generic complete eigenstructures in the set of $m \times m$ complex {\em general (unstructured)} matrix polynomials of odd grade $d$ and rank at most $r$.

There are two reasons for restricting our attention to skew-symmetric matrix polynomials of odd grade in this paper. The first one is that the general (that is, without imposing any rank constraint) sets of skew-symmetric matrix pencils and of skew-symmetric matrix polynomials of fixed odd grade $d$ have been deeply studied in the literature and their stratification hierarchies are well understood \cite{Dmyt15,DmKa14}, while similar results for other classes of structured pencils and matrix polynomials are not yet known. The second reason is that skew-symmetric pencils and matrix polynomials appear in a number of interesting applications, e.g., analysis of passive velocity field controllers \cite{LiHo99}, bi-Hamiltonian systems \cite{Olve91}, multisymplectic PDEs \cite{BrRe01}, investigation of a product of two skew-symmetric matrices and particular symmetric factorizations of skew-symmetric rational matrices \cite{MMMM13}.


The paper is organized as follows. Section \ref{pencils} includes some basic results on general and skew-symmetric pencils, which allow us to describe in Section \ref{sec.main1} the set of skew-symmetric matrix pencils with rank at most $2 w$ as the closure of the matrix pencils that have a single and simple skew-symmetric Kronecker canonical form. Section \ref{sect.prempolys} presents a number of preliminary results on general and skew-symmetric matrix polynomials that combined with the results for pencils of Section \ref{sec.main1} allow us to prove our main result for skew-symmetric matrix polynomials of fixed odd grade and fixed rank, i.e., Theorem \ref{mainth} in Section \ref{sec.main}. Finally, the conclusions, some lines of future research, and some difficulties for extending the results of this paper are discussed in Section \ref{sect.conclusions}.

The reader should bear in mind throughout this paper that all the matrices that we consider have complex entries, and that the rank of any pencil or any matrix polynomial is defined as the largest size of the minors that are not identically zero scalar polynomials \cite{Gant59}, which is sometimes called the normal rank \cite{EdEK97}. However, for brevity, we do not use the word ``normal'' in this paper.

\section{Preliminaries on skew-symmetric matrix pencils}
\label{pencils}

We start by recalling the Kronecker canonical form (KCF) of general matrix pencils and the canonical form of skew-symmetric matrix pencils under congruence. Define
$\overline{\mathbb C} := \mathbb C \cup  \infty$.
For each $k=1,2, \ldots $, define the $k\times k$
matrices
\begin{equation*}\label{1aa}
J_k(\mu):=\begin{bmatrix}
\mu&1&&\\
&\mu&\ddots&\\
&&\ddots&1\\
&&&\mu
\end{bmatrix},\qquad
I_k:=\begin{bmatrix}
1&&&\\
&1&&\\
&&\ddots&\\
&&&1
\end{bmatrix},
\end{equation*}
and for each $k=0,1, \ldots $, define the $k\times
(k+1)$ matrices
\begin{equation*}
F_k :=
\begin{bmatrix}
0&1&&\\
&\ddots&\ddots&\\
&&0&1\\
\end{bmatrix}, \qquad
G_k :=
\begin{bmatrix}
1&0&&\\
&\ddots&\ddots&\\
&&1&0\\
\end{bmatrix}.
\end{equation*}
All non-specified entries of $J_k(\mu), I_k, F_k,$ and $G_k$ are zeros.

An $m \times n$ matrix pencil $\lambda A - B$ is called {\em strictly equivalent} to $\lambda C - D$ if and only if there are non-singular matrices $Q$ and $R$ such that $Q^{-1}AR =C$ and $Q^{-1}BR=D$.

\begin{theorem}{\rm \cite[Ch. XII, Sect. 4]{Gant59}}\label{kron}
Each $m \times n$ matrix pencil $\lambda A - B$ is strictly equivalent
to a direct sum, uniquely determined up
to permutation of summands, of pencils of the form
\begin{align*}
{\cal E}_k(\mu)&:=\lambda I_k - J_k(\mu), \text { in which } \mu \in \mathbb C, \quad {\cal E}_k(\infty):=\lambda J_k(0) - I_k, \\
{\cal L}_k&:=\lambda G_k - F_k, \quad \text{ and } \quad {\cal L}_k^T:=\lambda G_k^T - F^T_k.
\end{align*}
This direct sum is called the KCF of $\lambda A -B$.
\end{theorem}
\noindent  The regular part of $\lambda A - B$ consists of the blocks ${\cal E}_k(\mu)$ and ${\cal E}_k(\infty)$ corresponding to the finite and infinite eigenvalues, respectively. The singular part of $\lambda A - B$ consists of the blocks ${\cal L}_k$ and ${\cal L}_k^T$ corresponding to the column and row minimal indices, respectively. The number of blocks ${\cal L}_k$ (respectively, ${\cal L}_k^T$) in the KCF of $\lambda A - B$ is equal to the dimension of the right (respectively, left) rational null-space of $\lambda A - B$.

Define an {\it orbit} of $\lambda A - B$ under the
action of the group $GL_m(\mathbb C) \times GL_n(\mathbb C)$ on the space of all matrix pencils by strict equivalence as follows:
\begin{equation} \label{equorbit}
\orb^e (\lambda A - B) = \{Q^{-1} (\lambda A - B) R \ : \ Q \in GL_m(\mathbb C), R \in GL_n(\mathbb C)\}.
\end{equation}
The orbit of $\lambda A - B$ is
a manifold in the complex $2mn$ dimensional space.

The space of all $m\times n$ matrix pencils is denoted by $\PEN_{m \times n}$. A distance in $\PEN_{m \times n}$ can be defined with the Frobenius norm of complex matrices \cite{Highambook} as $d(\lambda A - B, \lambda C - D) := \sqrt{\|A-C\|_F^2 + \|B-D\|_F^2}$, which makes $\PEN_{m \times n}$ into a metric space. This metric allows us to consider closures of subsets of $\PEN_{m \times n}$, in particular, closures of orbits by strict equivalence, denoted by $\overline{\orb^e}(\lambda A - B)$. By using these concepts, the following result from \cite{Bole98},  see also \cite{EdEK99}, describes all the possible changes in the KCF of a matrix pencil under arbitrarily small perturbations.
If in the KCF of a matrix pencil the blocks ${\cal X}$ are changed to the blocks ${\cal Y}$ (${\cal X}$ and ${\cal Y}$ of the same size) we write ${\cal X} \rightsquigarrow {\cal Y}$.
\begin{theorem}{\rm \cite{Bole98}}  \label{erelations}
Let ${\cal P}_1$ and ${\cal P}_2$ be two matrix pencils in KCF. Then,  $\overline{\orb^e}({\cal P}_1) \supset \orb^e ({\cal P}_2)$
if and only if ${\cal P}_1$ can be obtained from ${\cal P}_2$ changing canonical blocks of ${\cal P}_2$ by applying a sequence of rules and each rule is one of the six types below:
\begin{enumerate}
\item ${\cal L}_{j-1} \oplus {\cal L}_{k+1} \rightsquigarrow {\cal L}_j \oplus {\cal L}_k$, $1\le j \le k;$
\item ${\cal L}_{j-1}^T \oplus {\cal L}_{k+1}^T \rightsquigarrow {\cal L}_j^T \oplus {\cal L}_k^T$, $1\le j \le k;$
\item ${\cal L}_{j} \oplus {\cal E}_{k+1}(\mu) \rightsquigarrow {\cal L}_{j+1} \oplus {\cal E}_k(\mu)$, $j,k=0,1,2, \dots$ and $\mu \in \overline{\mathbb C};$
\item ${\cal L}_{j}^T \oplus {\cal E}_{k+1}(\mu) \rightsquigarrow {\cal L}_{j+1}^T \oplus {\cal E}_k(\mu)$, $j,k=0,1,2, \dots$ and $\mu \in \overline{\mathbb C};$
\item ${\cal E}_{j}(\mu) \oplus {\cal E}_{k}(\mu) \rightsquigarrow {\cal E}_{j-1}(\mu) \oplus {\cal E}_{k+1}(\mu)$, $1\le j \le k$ and $\mu \in \overline{\mathbb C};$
\item ${\cal L}_{p} \oplus {\cal L}_{q}^T \rightsquigarrow \bigoplus_{i=1}^t{\cal E}_{k_i}(\mu_i)$, if $p+q+1= \sum_{i=1}^t k_i$ and $\mu_i \neq \mu_{i'}$ for $i \neq i', \mu_i \in \overline{\mathbb C}.$
\end{enumerate}
Observe that in the rules above any block ${\cal E}_0 (\mu)$ should be understood as the empty matrix.
\end{theorem}

An $n \times n$ matrix pencil $\lambda A - B$ is called {\it congruent} to $\lambda C - D$ if there is a non-singular matrix $S$ such that $S^{T}AS =C$ and $S^{T}BS=D$.
\hide{
The {\it dimension} of $\orb_{\lambda A - B}^e$
is the dimension of its tangent space
\begin{equation*}\label{taneqsp}
\tsp^e (\lambda A + B):=\{\lambda (XA-AY) + (XB - BY):
X\in{\mathbb
C}^{m\times m}, Y\in{\mathbb
C}^{n\times n}\}
\end{equation*}
at the point $\lambda A - B$, $\dim \tsp^e (\lambda A - B)$.
The orthogonal complement to $\tsp^e(\lambda A - B)$, with respect to the Frobenius inner product
\begin{equation}\label{innerprod}
\langle \lambda A - B,\lambda C + D \rangle=\tr (AC^*+BD^*),
\end{equation}
is called the normal space to this orbit. The dimension of the normal space is the {\it codimension} of $\orb^e(\lambda A - B)$, denoted $\cod \orb^e(\lambda A - B)$, and is equal to $2mn$ minus the dimension of $\orb^e(\lambda A - B)$.
Explicit expressions for the codimensions of strict equivalence orbits are presented in \cite{DeEd95}.
}
In the following theorem we recall the canonical form under congruence of skew-symmetric matrix pencils, i.e., those satisfying $(\lambda A - B)^T = -(\lambda A - B)$.
\begin{theorem}{\rm \cite{Thom91} }\label{lkh}
Each skew-symmetric $n \times n$ matrix pencil $\lambda A - B$ is congruent
to a direct sum,
determined uniquely up
to permutation of
summands, of pencils of
the form
\begin{align*}
{\cal H}_{h}(\mu)&:=
\lambda
\begin{bmatrix}0&I_h\\
-I_h &0
\end{bmatrix} -
\begin{bmatrix}0&J_h(\mu)\\
-J_h(\mu)^T &0
\end{bmatrix}
,\quad \mu \in\mathbb C,\\
{\cal K}_k&:=
\lambda
\begin{bmatrix}0&J_k(0)\\
-J_k(0)^T &0
\end{bmatrix} -
\begin{bmatrix}0&I_k\\
-I_k&0
\end{bmatrix},\\
{\cal M}_m&:=
\lambda
\begin{bmatrix}0&G_m\\
-G_m^T&0
\end{bmatrix} -
\begin{bmatrix}0&F_m\\
-F_m^T &0
\end{bmatrix}.
\end{align*}
This direct sum is called the skew-symmetric KCF of $\lambda A-B$. Observe that the block ${\cal M}_0$ is just the $1 \times 1$ zero matrix.
\end{theorem}
\noindent Similarly to KCF,  the regular part of $\lambda A - B$ consists of the blocks ${\cal H}_{h}(\mu)$ and ${\cal K}_k$ corresponding to the finite and infinite eigenvalues, respectively. The singular part of $\lambda A - B$ consists of the blocks ${\cal M}_m$ corresponding to the right (column) and left (row) minimal indices, which are equal in the case of skew-symmetric matrix pencils.
Define an {\it orbit} of $\lambda A - B$ under the
action of the group $GL_n(\mathbb C)$ on the space of all skew-symmetric matrix pencils by congruence as follows:
\begin{equation} \label{equorbit}
\orb^c (\lambda A - B) = \{S^{T} (\lambda A - B) S \ : \ S \in GL_n(\mathbb C)\}.
\end{equation}
The orbit of $\lambda A - B$ is a manifold in the complex $n^2-n$ dimensional space of skew-symmetric pencils.

The fact that two skew-symmetric matrix pencils are congruent if and only if they have the same skew-symmetric KCF follows immediately from Theorem \ref{lkh}, as well as the well-known property that skew-symmetric matrix pencils have always even rank.

\section{Generic skew-symmetric matrix pencils with bounded rank} \label{sec.main1}

In this section we state and prove our main results about skew-symmetric matrix pencils, in particular, we will find the most generic congruence orbit of skew-symmetric matrix pencils of a fixed rank in Theorem \ref{th:improvedDeDo}, show its irreducibility in Lemma \ref{zariskilemma}, and count its codimension in \eqref{codimcomp}. These results will be used in Section \ref{sec.main} for obtaining similar results about skew-symmetric matrix polynomials. In order to keep the arguments concise, we often use expressions as ``the pencil ${\cal P}_1$ is more generic than the pencil ${\cal P}_2$'', whose precise meaning is that $\overline{\orb^e} ({\cal P}_1) \supset \orb^e ({\cal P}_2)$ or $\overline{\orb^c} ({\cal P}_1) \supset \orb^c ({\cal P}_2)$, depending of the context. In this language, the most generic skew-symmetric pencil with rank $2w$ is a pencil such that the closure of its congruence orbit includes the congruence orbit of any other skew-symmetric pencil with different KCF and with rank at most $2w$. We denote the vector space of skew-symmetric matrix pencils of size $n\times n$ by $\PEN_{n\times n}^{ss}$. A distance in $\PEN_{n\times n}^{ss}$ can be defined as $d(\lambda A - B, \lambda C - D) := \sqrt{\|A-C\|_F^2 + \|B-D\|_F^2}$, as in the case of general pencils. With this distance at hand, we can consider closures of subsets in $\PEN_{n\times n}^{ss}$, as well as any other topological concept.

\begin{theorem} \label{th:improvedDeDo}
Let $n$ and $w$ be integers such that $n \geq 2$ and $2\leq 2w \leq n-1$.
The set of $n \times n$ complex skew-symmetric matrix pencils with rank at most $2w$ is a closed subset of $\PEN^{ss}_{n \times n}$ equal to $\overline{\orb^c}({\cal W})$, where
\begin{equation}\label{max}
{\cal W} =\diag(\underbrace{{\cal M}_{\alpha+1},\hdots,{\cal M}_{\alpha+1}}_{s},
\underbrace{{\cal M}_{\alpha},\hdots,{\cal M}_{\alpha}}_{n-2w-s})\,
\end{equation}
with $\alpha= \lfloor w/(n-2w) \rfloor$ and
$s= w \, \mathrm{mod}\, (n-2w)$.
\hide{
Let us define, in the set of $n\times n$ complex skew-symmetric matrix pencils with
rank $2r_1$, the following skew-symmetric canonical form:
\begin{equation}\label{max}
{\cal W} (\lambda)=\diag(\underbrace{M_{\alpha+1},\hdots,M_{\alpha+1}}_{s},
\underbrace{M_{\alpha},\hdots,M_{\alpha}}_{n-2r_1-s})\,
\end{equation}
where $\alpha= \lfloor r_1/(n-2r_1) \rfloor$ and
$s= r_1 \, \mathrm{mod}\, (n-2r_1)$.
Then,
\begin{enumerate}
\item[\rm (i)] $\overline{\orb^c}({\cal M}) \supseteq\overline{\orb^c}({\cal X})$ for every $n_1\times n_1$ skew-symmetric matrix pencil ${\cal X}(\lambda)$ with rank at most $2r_1$.
\item[\rm (ii)] The set of $n_1 \times n_1$ complex matrix pencils with rank
at most $2r_1$ is a closed  subset of $\skewPEN_{n_1 \times n_1}$ equal to $\overline{\orb^e}({\cal M} )$.
\end{enumerate}
}
\end{theorem}
\begin{proof} Taking into account \cite[Lem. 3.8]{DmKa14} (or the stronger result \cite[Thm.~3.1]{DmKa14}), in this proof we work with KCF rather than with the skew-symmetric KCF under congruence of Theorem \ref{lkh}, but we always apply rules from Theorem \ref{erelations} in pairs, such that the corresponding change of the skew-symmetric KCF is obvious.

Each $n \times n$ skew-symmetric matrix pencil of the rank $2 r_1$ less than or equal to $2w$ has the following KCF:
 \begin{equation}\label{anyKCF}
\diag({\cal L}_{\gamma_1},\hdots,{\cal L}_{\gamma_{n-2r_1}}, {\cal L}^T_{\gamma_1},\hdots,{\cal L}^T_{\gamma_{n-2r_1}}, {\cal J}, {\cal J}),
\end{equation}
since left and right singular blocks are paired up according to Theorem \ref{lkh}, the number of left singular blocks is equal to the dimension of the rational left null-space of the pencil, i.e., is equal to $n-2r_1$, and  the blocks ${\cal J}$ of the regular part are also paired up.

Note that KCF of a generic skew-symmetric matrix pencil of rank less than or equal to $2w$ can not contain any regular part (otherwise rules 3 and 4 from Theorem \ref{erelations} can be applied to obtain a more generic skew-symmetric pencil). Moreover, the rank of any of the most generic skew-symmetric matrix pencils of rank at most $2w$ is exactly $2w$ (otherwise, there is an excess of singular blocks and rule 6 from Theorem \ref{erelations} can be applied twice to obtain a more generic skew-symmetric matrix pencil with strictly larger rank). Therefore any of the most generic skew-symmetric matrix pencils of rank at most $2w$ must have KCF consisting of  $n-2w$ blocks ${\cal L}$ and $n-2w$ blocks ${\cal L}^T$:
 \begin{equation}\label{anymKCF}
\diag({\cal L}_{\gamma_1},\hdots,{\cal L}_{\gamma_{n-2w}}, {\cal L}^T_{\gamma_1},\hdots,{\cal L}^T_{\gamma_{n-2w}}).
\end{equation}
To each pencil \eqref{anymKCF} we can apply only rules 1 and 2 from Theorem \ref{erelations} in order to get more generic skew-symmetric matrix pencils of rank at most $2w$, since applying rule 6 would increase the rank and rules 3--5 involve regular blocks. Thus we can only apply rule 1 to  the blocks ${\cal L}$ of \eqref{anymKCF} (and, simultaneously for preserving the skew-symmetric structure of the KCF, rule 2 to  the blocks ${\cal L}^T$ of \eqref{anymKCF}, respectively), which corresponds to finding the most generic $w \times (n-w)$ matrix pencil (and the most generic $(n-w) \times w$ matrix pencil, respectively). For the latter we refer to, e.g., \cite[Thm. 2.6]{DmDo16}, see also \cite{DeEd95,EdEK99,VanD79}\footnote{An equivalent, and more self-contained, argument is the following: since rule 1 (and rule 2) in Theorem \ref{erelations} can be applied to get more generic skew-symmetric pencils as long as there are blocks ${\cal L}_{\gamma_i}$ and ${\cal L}_{\gamma_j}$ (and ${\cal L}_{\gamma_i}^T$ and ${\cal L}_{\gamma_j}^T$) with $|\gamma_i - \gamma_j| \geq 2$, the KCF of the most generic skew-symmetric matrix pencils of rank at most $2w$ must have $n-2w$ blocks ${\cal L}_{\gamma_i}$ and $n-2w$ blocks ${\cal L}_{\gamma_i}^T$, $i=1,2, \ldots, n-2w$, with $|\gamma_i - \gamma_j| \leq 1$, which is only possible if the KCF is \eqref{maxKCF}.}. Therefore the most generic skew-symmetric matrix pencil of rank at most $2w$ has the KCF
\begin{equation}\label{maxKCF}
\diag(\underbrace{{\cal L}_{\alpha+1},\hdots,{\cal L}_{\alpha+1}}_{s},
\underbrace{{\cal L}_{\alpha},\hdots,{\cal L}_{\alpha}}_{n-2w-s},\underbrace{{\cal L}^T_{\alpha+1},\hdots,{\cal L}^T_{\alpha+1}}_{s},\underbrace{{\cal L}^T_{\alpha},\hdots,{\cal L}^T_{\alpha}}_{n-2w-s}),\,
\end{equation}
or equivalently, has the skew-symmetric KCF \eqref{max}.
\end{proof}
\begin{remark}
The KCF in \eqref{maxKCF} is also one of the most generic KCFs of the $n \times n$ general matrix pencils of rank less than or equal to $2w$, obtained in \cite[Thm. 3.2]{DeDo08}. Therefore proving the genericity of \eqref{max} could have been done using \cite[Thm. 3.2]{DeDo08} and \cite[Thm. 3.1]{DmKa14} but we would still have to explain why there are no more generic pencils (essentially, to repeat the proof of Theorem \ref{th:improvedDeDo}).
\end{remark}

Observe that Theorem \ref{th:improvedDeDo} does not cover the case $2w =n$, which generically corresponds to regular skew-symmetric matrix pencils with even number of rows (and columns). In this easy case, the generic pencils clearly have $w$ different eigenvalues which are paired up according to Theorem \ref{lkh}, i.e. the generic KCF is $\diag({\cal H}_{1}(\mu_{1}),\hdots,{\cal H}_{1}(\mu_{w}))$, where $\mu_i \neq \mu_j$ if $i \neq j$, and $i,j\in \{1, 2, \dots , w\}$.

Similarly to results included in \cite{DeDo08,Wate84} for closures of orbits under strict equivalence, next we show that for any $n\times n$ skew-symmetric pencil ${\cal P}$ the set $\overline{\orb^c} ({\cal P})$ (where the closure is taken in principle in the Euclidean metric defined before Theorem \ref{th:improvedDeDo}) is irreducible in the Zariski topology of $\PEN_{n\times n}^{ss}$. The proof of such result also establishes that the closures of $\orb^c ({\cal P})$ in the Euclidean and in the Zariski topologies are equal. In particular, this irreducibility implies the connectivity of $\overline{\orb^c}({\cal P})$ in both the Zariski and Euclidean topologies (recall that we work over $\mathbb C$).  Note also that by \cite[Closed Orbit Lemma, p. 53]{Bore91}, $\orb^c ({\cal P})$ is open in its closure again in both the Zariski and Euclidean topologies. These properties hold, of course, for the generic pencil ${\cal W}$ in \eqref{max}. The proof of Lemma \ref{zariskilemma} essentially repeats the proof of \cite[Lem.~3.4]{DeDo08} where the analogous result for the general matrix pencils under strict equivalence is derived.
\begin{lemma} \label{zariskilemma}
The closures in the Euclidean topology and in the Zariski topology of the congruence orbit $\orb^c ({\cal P})$ of any $n \times n$ complex skew-symmetric matrix pencil  ${\cal P}$ are equal in $\PEN_{n\times n}^{ss}$. Moreover, $\overline{\orb^c} ({\cal P})$ is an irreducible manifold in the Zariski topology of $\PEN_{n\times n}^{ss}$.
\end{lemma}
\begin{proof}
We identify the space of skew-symmetric matrix pencils of size $n \times n$ with $\mathbb C^{n^2-n}$ (for the pencil ${\cal P} = \lambda A - B$ we have $(n^2-n)/2$ parameters in each of $A$ and $B$). All the nonsingular matrices $S \in \mathbb C^{n \times n}$ form a dense open set $U$ of $\mathbb C^{n^2}$ both in the Zariski and in the Euclidean topology. We consider the polynomial mapping $\phi_{{\cal M}}$ from $\mathbb C^{n^2}$ to $\mathbb C^{n^2-n}$ defined by sending $S$ to $\lambda S^TAS - S^TBS$. Thus $\phi_{{\cal M}}(U) = \orb^c ({\cal P})$. Notably, in any topology, for a set $V$ and a continuous mapping $\phi$, we have $\phi (\overline{V}) \subset \overline{\phi(V)} $, which implies $\overline{\phi (\overline{V})} \subset \overline{\phi(V)} $. Therefore, we have
$
\overline{\phi_{{\cal M}}(U)} = \overline{\phi_{{\cal M}}(\mathbb C^{n \times n})},
$
where the closures can be taken either in the Zariski or the Euclidean topology, since in this particular  case they are equal.
The remaining part of the proof is to note that $\overline{\phi_{{\cal M}}(\mathbb C^{n^2})}$ is an irreducible set in the Zariski topology, by \cite[Sect. 1]{Wate84}, and is equal to $\overline{\orb^c} ({\cal P})$, by the discussion above.
\end{proof}

For an $n\times n$ skew-symmetric pencil $\lambda A - B$, define the {\it dimension} of $\orb^c (\lambda A - B)$
to be the dimension of the tangent space to this orbit
\begin{equation*}\label{taneqsp}
\tsp_{\lambda A - B}^c:=\{\lambda (X^TA+AX) - (X^TB + BX):
X\in{\mathbb
C}^{n\times n}\}
\end{equation*}
at the point $\lambda A - B$.
The orthogonal complement (with respect to the Frobenius inner product) to $\tsp_{\lambda A - B}^c$,
is called the normal space to $\orb^c (\lambda A - B)$ at the point $\lambda A - B$. The dimension of the normal space is the {\it codimension} of the congruence orbit of $\lambda A - B$ and is equal to $n(n-1)$ minus the dimension of the congruence orbit of $\lambda A - B$.
Explicit expressions for the codimensions of congruence orbits of skew-symmetric pencils in $\PEN_{n\times n}^{ss}$ are presented in \cite{DmKS13} and implemented in the MCS (Matrix Canonical Structure) Toolbox \cite{DmJK13, Joha06}.
Using \cite[Thm.~3]{DmKS13}, see also \cite[Thm. 2.8]{DmJK13}, applied to the skew-symmetric pencil ${\cal W}$ in \eqref{max}, the codimension of $\orb^c ({\cal W})$ in $\PEN_{n\times n}^{ss}$ is $\cod \orb^c ({\cal W}) = \sum_{i<j} (2 \max \{m_i,m_j \} + \varepsilon_{ij})$, where $m_i$ and $m_j$ are the indices of the ${\cal M}$ blocks (either $\alpha$ or $\alpha +1$), $\varepsilon_{ij} = 2$ if $m_i=m_j$ and $\varepsilon_{ij} = 1$ otherwise, resulting in:
\begin{equation} \label{codimcomp}
\begin{aligned}
&\cod \orb^c ({\cal W}) = (n-2w-s)(n-2w-s-1)(2\alpha +2)/2  \\ 
&\phantom{a}  + s(s-1)(2(\alpha+1) +2)/2 + s(n-2w-s)(2(\alpha+1) +1)\\
&\phantom{a} = (n-2w-s)(n-2w-s-1)(\alpha +1) + s(s-1)(\alpha+1)  + s(s-1) \\
&\phantom{a} + 2s(n-2w-s)(\alpha+1) + s(n-2w-s)\\
&\phantom{a} = (\alpha +1) ((n-2w-s)(n-2w-1) + s(n-2w-1)) + s(n-2w-1) \\
&\phantom{a} =(n-2w)(n-2w-1)(\alpha +1)+s(n-2w-1) \\
&\phantom{a} =(n-2w-1)((n-2w)\alpha +s + n-2w)=(n-2w-1)(n-w).
\end{aligned}
\end{equation}

\hide{
\begin{theorem} \label{th-codimensions}
Let $n$ and $w$ be integers such that $n \geq 2$ and $1 \leq 2w \leq n-1$ and let ${\cal W}(\lambda)$, be the matrix polynomials with the complete eigenstructures \eqref{max}.
Then the codimension of $O^c({\cal W}(\lambda))$ in $\PEN_{n\times n}^{ss}$ is
$(n-2w-1)(n-w).$

\hide{  {2. The codimension in $\POL_{d,m\times n}$ of the set of $m \times n$ complex matrix polynomials with rank at most $r$ and grade $d$ is equal to}

{\red  {(i)  $(n-r)(m(d+1)-r)$ if $m\geq n$, and}}

{\red  {(ii) $(m-r)(n(d+1)-r)$  if $m\leq n$.}}
}
\end{theorem}

\begin{proof}  The codimension of $O^c({\cal W}(\lambda))$ can be computed using  \cite[Thm.~3]{DmKS13} as follows:
\begin{align*}
&(n-r)(2m+n(d-1)-r) + ((n(d-1)+r) -rd + a) (m-n)\\
&\phantom{a}=(n-r)(2m-r)+(n-r)n(d-1)+n(d-1)(m-n)-r(d-1)(m-n)+ a(m-n) \\
&\phantom{a} =(n-r)(2m-r)+(n-r)n(d-1)+(n-r)(d-1)(m-n)+ a(m-n) \\
&\phantom{a}=(n-r)(2m-r+n(d-1)+(d-1)(m-n))+ a(m-n) \\
&\phantom{a} =(n-r)((d+1)m -r)+ a(m-n).
\end{align*}
\end{proof}
}

\section{Preliminaries on skew-symmetric matrix polynomials} \label{sect.prempolys}
We consider skew-symmetric $m\times m$ matrix polynomials $P(\lambda)$ of grade $d$, i.e., of degree less than or equal to $d$, over~$\mathbb C$:
\begin{equation*}
P(\lambda) = \lambda^{d}A_{d} + \dots +  \lambda A_1 + A_0,
\quad A_i^T=-A_i, \ A_i \in \mathbb C^{m \times m} \ \text{for } i=0, \dots, d.
\end{equation*}
Denote the vector space of the $m \times m$ skew-symmetric matrix polynomials of the  grade $d$ by $\POL_{d, m\times m}^{ss}$.
Observe that $\POL_{1, n\times n}^{ss}$ is $\PEN_{n\times n}^{ss}$.
If there is no risk of confusion we will write $\POL$ instead of $\POL_{d, m\times m}^{ss}$. As in the case of pencils, see Sections \ref{pencils} and \ref{sec.main1}, by using  the Frobenius matrix norm of complex matrices \cite{Highambook} a distance on $\POL_{d, m\times m}^{ss}$ is defined as $d(P,P') = \left( \sum_{i=0}^d || A_i - A'_i ||_F^2 \right)^{\frac{1}{2}}$, making $\POL_{d, m\times m}^{ss}$ to a metric space  with the induced Euclidean topology. For convenience,  define the Frobenius norm of the matrix polynomial $P$ as $||P(\lambda)||_F = \left( \sum_{i=0}^d || A_i ||_F^2 \right)^{\frac{1}{2}}$.

Two matrix polynomials $P(\lambda)$ and $Q(\lambda)$ are called {\it unimodularly congruent} if $F(\lambda)^T P(\lambda) F(\lambda)=Q(\lambda)$ for some unimodular matrix polynomial $F(\lambda)$ (i.e. $\det F(\lambda) \in \mathbb C \backslash \{0\}$), see also \cite{MMMM13}.
\begin{theorem}{\rm \cite{MMMM13}} \label{tsmiths}
Let $P(\lambda)$ be a skew-symmetric $m\times m$ matrix polynomial. Then there exists $r \in \mathbb N$ with $2r \le m$ and a unimodular matrix polynomial $F(\lambda)$ such that
\begin{equation*}
F(\lambda)^T P(\lambda) F(\lambda) =
\begin{bmatrix}
0 & g_1(\lambda)\\
-g_1(\lambda) & 0
\end{bmatrix} \oplus
\dots
\oplus
\begin{bmatrix}
0 & g_r(\lambda)\\
-g_r(\lambda) & 0
\end{bmatrix} \oplus
0_{m-2r}=:S(\lambda),
\end{equation*}
where $g_j$ is monic for $j=1, \dots, r$ and $g_j(\lambda)$ divides $g_{j+1}(\lambda)$ for $j=1, \dots, r-1$. Moreover, the canonical form $S(\lambda)$ is unique.
\end{theorem}
Similarly to the unimodular congruence, two matrix polynomials $P(\lambda)$ and $Q(\lambda)$ are called {\it unimodularly equivalent} if $U(\lambda) P(\lambda) V(\lambda)=Q(\lambda)$ for some unimodular matrix polynomials \ $U(\lambda)$ \ and \ $V(\lambda)$ \ (i.e. $\det U(\lambda), \det V(\lambda) \in \mathbb C \backslash \{0\}$). Note that the canonical form in Theorem \ref{tsmiths} is the skew-symmetric version of the well-known Smith form for matrix polynomials under unimodular equivalence \cite{Gant59,MMMM13}.

Define the (normal) {\it rank} of the matrix polynomial $P(\lambda)$ to be the integer $2r$ from Theorem \ref{tsmiths}. All $g_j(\lambda), j=1, \dots, r,$ are called {\it invariant polynomials} of $P(\lambda)$, and each of them can be uniquely factored as
$$
g_j(\lambda) = (\lambda - \alpha_1)^{\delta_{j1}} \cdot
(\lambda - \alpha_2)^{\delta_{j2}}\cdot \ldots \cdot
(\lambda - \alpha_{l_j})^{\delta_{jl_j}},
$$
where $l_j \ge 0, \ \delta_{j1}, \dots, \delta_{jl_j} > 0$ are integers. If $l_j=0$ then $g_j(\lambda)=1$. The complex numbers $\alpha_1, \dots, \alpha_{l_j}$ are finite eigenvalues of $P(\lambda)$. The {\it elementary divisors} of $P(\lambda)$ associated with each finite eigenvalue $\alpha_{k}$ is the collection of factors $(\lambda - \alpha_{k})^{\delta_{jk}}$ (possibly with repetitions).

If zero is an eigenvalue of $\rev P(\lambda):= \lambda^dP(1/\lambda)$ then we say that $\lambda = \infty$ is an eigenvalue of the matrix polynomial $P(\lambda)$ of grade $d$. The elementary divisors $\lambda^{\gamma_k}, \gamma_k > 0,$ for the zero eigenvalue of $\rev P(\lambda)$ are the elementary divisors associated with the infinite eigenvalue of $P(\lambda)$.

For an $m\times n$ matrix polynomial $P(\lambda)$, define the left and right null-spaces, over the field of rational functions $\mathbb C(\lambda)$, as follows:
\begin{align*}
{\cal N}_{\rm left}(P)&:= \{y(\lambda)^T \in \mathbb C(\lambda)^{1 \times m}: y(\lambda)^TP(\lambda) = 0_{1\times n} \}, \\
{\cal N}_{\rm right}(P)&:= \{x(\lambda) \in \mathbb C(\lambda)^{n\times 1}: P(\lambda)x(\lambda) = 0_{m\times 1}\}.
\end{align*}
Each subspace ${\cal V}$ of $\mathbb C(\lambda)^n$ has bases consisting entirely of vector polynomials.
A basis of ${\cal V}$ consisting of vector polynomials whose sum of degrees is minimal among all bases of ${\cal V}$ consisting of vector polynomials is called a minimal basis of ${\cal V}$. The ordered list of degrees of the vector polynomials in any minimal basis of ${\cal V}$ is always the same. These degrees are called the minimal indices of ${\cal V}$ \cite{Forn75,Kail80}. This allows us to define the minimal indices of a matrix polynomial:
let the sets $\{y_1(\lambda)^T,...,y_{m-r}(\lambda)^T\}$ and $\{x_1(\lambda),...,x_{n-r}(\lambda)\}$ be minimal bases of ${\cal N}_{\rm left}(P)$ and ${\cal N}_{\rm right}(P)$, respectively, ordered so that $0 \le \deg(y_1) \le \dots \le \deg(y_{m-r})$ and $0\le \deg(x_1) \le \dots \le \deg(x_{n-r})$. Let $ \eta_k = \deg(y_k)$ for $k=1, \dots , m-r$ and $ \varepsilon_k = \deg(x_k)$ for $k=1, \dots , n-r$. Then the scalars $0 \le  \eta_1 \le \eta_2 \le  \dots \le \eta_{m-r}$ and $0 \le  \varepsilon_1 \le \varepsilon_2 \le  \dots \le \varepsilon_{n-r}$ are, respectively, the {\it left} and {\it right minimal indices} of~$P(\lambda)$. Note also that for a skew-symmetric matrix polynomial we have that  $x_i(\lambda) = y_i(\lambda)$ and thus $\eta_i = \varepsilon_i,$ for $i = 1, \dots , m-r$.

Define the {\it complete eigenstructure} of a matrix polynomial $P(\lambda)$ as all the finite and infinite eigenvalues, the corresponding elementary divisors, and the left and right minimal indices of $P(\lambda)$. The set of matrix polynomials of the same size, grade, and with the same complete eigenstructure as $P(\lambda)$ is called an {\it orbit} of $P(\lambda)$, denoted $\orb(P)$. In this paper, if $P(\lambda)$ is skew-symmetric, then $\orb(P)$ contains all the  skew-symmetric polynomials of the same size, grade, and with the same complete eigenstructure as $P(\lambda)$.


A matrix pencil ${\cal L}_{P(\lambda)}$ is called a {\it linearization} of a matrix polynomial $P(\lambda)$ if they have the same finite elementary divisors, the same number of left minimal indices, and the same number of right minimal indices \cite{DeDM14}. If in addition, $\rev {\cal L}_{P(\lambda)}$ is a linearization of $\rev P(\lambda)$ then ${\cal L}_{P(\lambda)}$ is called a {\it strong linearization} of $P(\lambda)$ and, then, ${\cal L}_{P(\lambda)}$ and $P(\lambda)$ have also the same infinite elementary divisors.  Linearizations are powerful tools for investigation of matrix polynomials \cite{DeDM14,GoLR09}.

From now on we restrict to skew-symmetric matrix polynomials of odd grades. The reason is that there is no skew-symmetric linearization-template (i.e., a skew-symmetric companion form in the language of \cite[Sects. 5 and 7]{DeDM14}) for skew-symmetric matrix polynomials of even grades \cite{DeDM14,MMMM13}.

The following pencil-template is known to be a skew-symmetric strong linearization of all the skew-symmetric $m\times m$ matrix polynomials $P(\lambda)$ of odd grade $d$ \cite{MMMM13}, see also \cite{AnVo04,MMMM10}:
\begin{align*}
{\cal L}_{P(\lambda)}(i,i)&=\begin{cases}
\lambda A_{d-i+1} + A_{d-i} & \text{if } i \text{ is odd,}\\
0 & \text{if } i \text{ is even,}\\
\end{cases}\\
{\cal L}_{P(\lambda)}(i,i+1)&=\begin{cases}
- I_m & \text{if } i \text{ is odd,}\\
- \lambda I_m & \text{if } i \text{ is even,}\\
\end{cases} \quad
{\cal L}_{P(\lambda)}(i+1,i)=\begin{cases}
I_m & \text{if } i \text{ is odd,}\\
\lambda I_m & \text{if } i \text{ is even,}\\
\end{cases}
\end{align*}
where ${\cal L}_{P(\lambda)}(j,k)$ denotes an $m \times m$ matrix pencil which is at the position $(j,k)$ of the block pencil ${\cal L}_{P(\lambda)}$ and  $j,k=1, \dots , d$. The blocks of ${\cal L}_{P(\lambda)}$ in positions which are not specified above are zero. We rewrite this strong linearization template in  a matrix form:
\hide{
\begin{equation}
\label{linform}
{\cal L}_{P(\lambda)}=
\begin{bmatrix}
\lambda A_d + A_{d-1}& -I &&&&&\\
I &0&-\lambda I&&&&\\
&\lambda I &\ddots&\ddots&&&\\
&&\ddots&0&-\lambda I&&\\
&&&\lambda I&\lambda A_{3}+A_{2}&-I&\\
&&&&I&0&-\lambda I\\
&&&&&\lambda I&\lambda A_1+ A_0\\
\end{bmatrix}
\end{equation}
or }
\begin{equation}
\label{linform}
{\cal L}_{P(\lambda)}=
\lambda
\begin{bmatrix}
A_d&&&&&\\
&\ddots&\ddots&&&\\
&\ddots&0&-I&&\\
&&I&A_{3}&&\\
&&&&0&-I\\
&&&&I&A_1\\
\end{bmatrix} - \begin{bmatrix}
-A_{d-1}&I&&&&\\
-I&0&\ddots&&&\\
&\ddots&\ddots&&&\\
&&&-A_{2}&I&\\
&&&-I&0&\\
&&&&&-A_0\\
\end{bmatrix}.
\end{equation}
For a skew-symmetric matrix polynomial, strong linearization \eqref{linform} preserves finite and infinite elementary divisors of $P(\lambda)$ but does not preserve the left and right minimal indices of $P(\lambda)$. Nevertheless, the relations between the minimal indices of a skew-symmetric matrix polynomial $P(\lambda)$ and its linearization \eqref{linform} are derived in \cite{Dmyt15}, see also \cite{DeDM10, DeDM12}.
\begin{theorem}{\rm \cite{Dmyt15}} \label{skewlin}
Let $P(\lambda)$ be a skew-symmetric $m\times m$ matrix polynomial of odd grade $d \ge 3$, and let ${\cal L}_{P(\lambda)}$ be its linerization \eqref{linform} given above. If $0 \le  \varepsilon_1 \le \varepsilon_2 \le~\dots \le \varepsilon_t$ are the right (=left) minimal indices of $P(\lambda)$ then
$$0 \le  \varepsilon_1 + \frac{1}{2}(d-1) \le \varepsilon_2 + \frac{1}{2}(d-1) \le  \cdots \le \varepsilon_t + \frac{1}{2}(d-1)$$
are the right (=left) minimal indices of ${\cal L}_{P(\lambda)}$.
\end{theorem}

The linearization ${\cal L}_{P(\lambda)}$ \eqref{linform} is crucial for obtaining the results in Section~\ref{sec.main}. Therefore we define the {\it generalized Sylvester space} consisting of the linearizations ${\cal L}_{P(\lambda)}$ of all the $m \times m$ skew-symmetric matrix polynomials of odd grade $d$:
\begin{equation}\label{gsyl}
\begin{aligned}
\GSYL^{ss}_{d,m\times m}= \{ {\cal L}_{P(\lambda)}  \ : P(\lambda) &\text{ are } m \times m \text{ skew-symmetric} \\
&\text{matrix polynomials of odd grade } d \}.
\end{aligned}
\end{equation}
If there is no risk of confusion we will write $\GSYL$ instead of $\GSYL^{ss}_{d, m\times m}$, specially in  explanations and proofs. The function $d({\cal L}_{P(\lambda)}=\lambda A - B, {\cal L}_{P'(\lambda)}=\lambda A' - B') := \left( ||A-A'||_F^2 + ||B-B'||_F^2 \right)^{\frac{1}{2}}$ mentioned in Sections \ref{pencils} and \ref{sec.main1} is a distance on $\GSYL$ and it makes $\GSYL$ a metric space.  Since $d({\cal L}_{P(\lambda)}, {\cal L}_{P'(\lambda)}) = d(P,P')$, there is a bijective isometry (and therefore a homeomorphism):
$$f: \POL^{ss}_{d, m\times m} \rightarrow \GSYL^{ss}_{d, m\times m} \quad \text{such that} \quad f: P \mapsto {\cal L}_{P}.$$
Next we define the orbit of the skew-symmetric linearizations of type \eqref{linform} of a fixed skew-symmetric matrix polynomial $P$
\begin{equation}\label{linorb}
\orb({\cal L}_{P}) = \{(S^{T}{\cal L}_{P(\lambda)} S) \in \GSYL^{ss}_{d, m\times m} \ : \ S \in GL_{n}(\mathbb C), \ n=md\}.
\end{equation}
We emphasize that all the elements of $\orb({\cal L}_{P})$ have the block structure of the elements of $\GSYL^{ss}_{d, m\times m}$.
Thus, in particular, $\orb(P) = f^{-1}(\orb({\cal L}_{P}))$, as a consequence of the properties of strong linearizations and Theorem \ref{skewlin}, and $\overline{\orb}(P) = f^{-1}(\overline{\orb}({\cal L}_{P}))$, as a consequence of $f$ being a homeomorphism. Moreover, we also have that for any $m \times m$ skew-symmetric matrix polynomials $P, Q$ of odd grade $d$, $\overline{\orb}(P) \supseteq \overline{\orb}(Q)$ if and only if $\overline{\orb}({\cal L}_{P}) \supseteq \overline{\orb}({\cal L}_{Q})$, where it is essential to note that the closures are taken in the metric spaces $\POL$ and $\GSYL$, respectively, defined above.  Note that similarly to the matrix pencil case, $\orb({\cal L}_{P})$ is open in its closure in the relative Euclidean topology, and so is $\orb(P)$ since $f$ is a homeomorphism.

We will use in Section \ref{sec.main} the fact that for any $m\times m$ skew-symmetric matrix polynomial $P(\lambda)$ a sufficiently small arbitrary skew-symmetric perturbation of the pencil ${\cal L}_{P(\lambda)}$ produces another pencil that although, in general, is not in $\GSYL^{ss}_{d, m\times m}$ is congruent to a pencil in $\GSYL^{ss}_{d, m\times m}$ that is very close to ${\cal L}_{P(\lambda)}$, see \cite{Dmyt15}. This type of results for the general matrix polynomials is presented in e.g., \cite{DLPVD15, JoKV13, VaDe83}.

\begin{theorem}[Theorems 8 and 9 in \cite{Dmyt15}]\label{deform}
Let $P(\lambda)$ be an $m \times m$ skew-symmetric  matrix polynomial with odd grade $d$ and let ${\cal L}_{P(\lambda)}$ be its skew-symmetric linearization~\eqref{linform}. If ${\mathcal E}$ is any arbitrarily small (entrywise) $md \times md$ skew-symmetric pencil, then there exist a nonsingular matrix $C$ and an arbitrarily small (entrywise) $m \times m$ skew-symmetric polynomial $F(\lambda)$ with grade $d$ such that
\begin{equation*}
\label{eq19}
C^T ({\cal L}_{P(\lambda)} +{\mathcal E}) C = {\cal L}_{P(\lambda) + F(\lambda)} \,.
\end{equation*}
\end{theorem}

It is interesting to emphasize that the extension of Theorem \ref{deform} for symmetric, Hermitian, skew-Hermitian, palindromic, anti-palindromic, and alternating matrix polynomials of odd grade and for the linearization \eqref{linform} follows immediately from the proof of Theorem 8 in \cite{Dmyt15}. An extension of Theorem \ref{deform} to all the structures mentioned above and to a wider class of structure-preserving companion linearizations which include \eqref{linform} can be obtained as a corollary of a much more general result recently obtained in \cite[Thm.~6.12]{DPVD16}.


\section{Generic skew-symmetric matrix polynomials with bounded rank and fixed grade}
\label{sec.main}

In this section we present the complete eigenstructure of the generic $m\times m$ skew-symmetric matrix polynomial of a fixed rank $2r$ and odd grade $d$. As in the case for general matrix polynomials \cite[Lem. 3.1]{DmDo16} we present the following lemma revealing a key relation between $\overline{\orb}({\cal L}_{P})$, where the closure is taken in $\GSYL^{ss}_{d,m\times m}$, and $\overline{\orb^c}({\cal L}_{P})$, where the closure is taken in $\PEN^{ss}_{n \times n}$, with $n = m d$.

\begin{lemma}\label{cl}
Let $P$ be an $m\times m$ skew-symmetric matrix polynomial with odd grade $d$ and ${\cal L}_{P}$ be its linearization \eqref{linform} then $\overline{\orb}({\cal L}_{P})= \overline{\orb^c}({\cal L}_{P}) \cap \GSYL^{ss}_{d,m\times m}$.
\end{lemma}
\begin{proof}
By definition $\orb({\cal L}_{P})= \orb^c({\cal L}_{P}) \cap \GSYL$ and thus $\overline{\orb}({\cal L}_{P})= \overline{\orb^c({\cal L}_{P}) \cap \GSYL}$  (the closure here is taken in the space $\GSYL$). For any ${\cal L}_{Q} \in \overline{\orb^c}({\cal L}_{P})\cap \GSYL$ there exists an arbitraly small  (entrywise) skew-symmetric pencil ${\mathcal E}$ such that $({\cal L}_{Q} +{\mathcal E}) \in \orb^c({\cal L}_{P})$. Therefore by Theorem \ref{deform},
there exists an arbitrarily small  (entrywise) skew-symmetric matrix polynomial $F$ with grade $d$ such that ${\cal L}_{Q+F} \in \orb^c({\cal L}_{P})$. Thus ${\cal L}_{Q}  \in \overline{\orb^c({\cal L}_{P}) \cap \GSYL}$, and  $\overline{\orb^c}({\cal L}_{P}) \cap \GSYL \subseteq\overline{\orb^c({\cal L}_{P}) \cap \GSYL}$.  Since $\overline{\orb^c({\cal L}_{P}) \cap \GSYL} \subseteq \overline{\orb^c}({\cal L}_{P}) \cap \GSYL$, we have that $\overline{\orb^c({\cal L}_{P}) \cap \GSYL}= \overline{\orb^c}({\cal L}_{P}) \cap \GSYL$, and the result is proved.
\end{proof}

Lemma \ref{cl} is the key to prove our main result about skew-symmetric matrix polynomials which uses and generalizes Theorem \ref{th:improvedDeDo}.

\begin{theorem} \label{mainth}
Let $m,r$ and $d$ be integers such that $m \geq 2$, $d \geq 1$ is odd, and $2 \leq 2r \leq (m-1)$.
The set of $m\times m$ complex skew-symmetric matrix polynomials of grade $d$ with rank at most $2r$ is a closed subset of $\POL_{d,m\times m}^{ss}$ equal to $\overline{\orb}(W)$, where $W$ is an $m \times m$ complex skew-symmetric matrix polynomial of degree exactly $d$ and rank exactly $2r$ with no elementary divisors at all, with $t$ left minimal indices equal to $(\beta +1)$ and with $(m-2r-t)$ left minimal indices equal to $\beta$, where $\beta = \lfloor rd / (m-2r) \rfloor$ and $t = rd \mod (m-2r)$, and with the right minimal indices equal to the left minimal indices.
\end{theorem}

\hide{
\begin{theorem} \label{mainth}
Let $m,r$ and $d$ be integers such that $m \geq 2$, $d \geq 1$ and $2 \leq 2r \leq (m-1)$. Define the following complete eigenstructures ${\cal K}$ of skew-symmetric matrix polynomials without elementary divisors at all, with left minimal indices equal to the right minimal indices and equal to $\beta$ and $\beta +1$, whose values and numbers are as follows:
\begin{equation}
\label{kcilist}
{\cal W}: \{\underbrace{\beta+1, \dots , \beta+1}_{t},\underbrace{\beta, \dots , \beta}_{m-2r-t}, \underbrace{\beta+1, \dots , \beta+1}_{t}, \underbrace{\beta, \dots , \beta}_{m-2r-t}\},
\end{equation}
where $\beta = \lfloor rd / (m-2r) \rfloor$ and $t = rd \mod (m-2r)$. Then,
\begin{itemize}
\item[(i)] There exists an $m \times m$ complex skew-symmetric matrix polynomial $K$ of degree exactly $d$ and rank exactly $2r$ with the complete eigenstructure ${\cal K}$;

\item[(ii)]  $\overline{\orb}({\cal K}) \supseteq\overline{\orb}({\cal X})$ for every $m \times m$ skew-symmetric matrix polynomial $X$ of grade $d$ with rank at most $2r$;

\item[(iii)]  The set of $m\times m$ complex skew-symmetric matrix polynomials of grade $d$ with rank at most $2r$ is a closed subset of $\POL_{d,m\times m}$ equal to $\overline{O}(K)$.
\end{itemize}
\end{theorem}
}
\begin{proof}
Note that Theorem \ref{mainth} for $d=1$ coincides with Theorem \ref{th:improvedDeDo}.

Denote the complete eigenstructure from the statement of Theorem \ref{mainth} (consisting only of the left and right minimal indices) by
\begin{equation}
\label{kcilist}
{\bf W}: \bigg\{\overbrace{\underbrace{\beta+1, \dots , \beta+1}_{t},\underbrace{\beta, \dots , \beta}_{m-2r-t}}^{\text{left minimal indices}}, \overbrace{\underbrace{\beta+1, \dots , \beta+1}_{t}, \underbrace{\beta, \dots , \beta}_{m-2r-t}}^{\text{right minimal indices}} \bigg\}.
\end{equation}
First we show that there exists an $m\times m$ skew-symmetric matrix polynomial $W$ of degree exactly $d$ and rank exactly $2r$ that has the complete eigenstructure~${\bf W}$ \eqref{kcilist}. By \cite[Thm. 3.3]{Dmyt15} it is enough to show that the sum of the left (or right) minimal indices of ${\bf W}$ is equal to $rd$:
\begin{align*}
&\sum_1^t (\beta +1) + \sum_1^{m-2r-t} \beta
= \sum_1^{m-2r} \beta + t = (m-2r) \lfloor rd/(m-2r) \rfloor + t = rd.
\end{align*}

For every $m \times m$ matrix polynomial $P$ of grade $d$ and rank at most $2r$, the linearization ${\cal L}_{P}$ has rank at most $2r+m(d-1)$, because ${\cal L}_{P}$ is unimodularly equivalent to $P \oplus I_{m(d-1)}$.
The linearization ${\cal L}_{W}$ of the matrix polynomial $W$ is an $md \times md$ skew-symmetric matrix pencil with the rank $m(d-1)+2r$ and by Theorem \ref{skewlin} the KCF of ${\cal L}_{W}$ is the direct sum of the following blocks:
\begin{equation} \label{cka}
\{\underbrace{{\cal L}_{\beta+\eta+1}, \dots , {\cal L}_{\beta+\eta+1}}_{t},\underbrace{{\cal L}_{\beta+\eta}, \dots , {\cal L}_{\beta+\eta}}_{m-2r-t}, \underbrace{{\cal L}_{\beta+\eta+1}^T, \dots , {\cal L}_{\beta+\eta+1}^T}_{t}, \underbrace{{\cal L}_{\beta+\eta}^T, \dots , {\cal L}_{\beta+\eta}^T}_{m-2r-t}\},
\end{equation}
where $\eta=\frac{1}{2}(d-1)$.
We show that the KCF of ${\cal L}_{W}$ coincides with the KCF of the most generic skew-symmetric matrix pencil ${\cal W}$ of rank $2w=m(d-1)+2r$ and size $n \times n$, where $n = m d$, given in Theorem \ref{th:improvedDeDo}:
\begin{equation} \label{pg}
\{\underbrace{{\cal L}_{\alpha+1}, \dots , {\cal L}_{\alpha+1}}_{s},\underbrace{{\cal L}_{\alpha}, \dots , {\cal L}_{\alpha}}_{n-2w-s}, \underbrace{{\cal L}_{\alpha+1}^T, \dots , {\cal L}_{\alpha+1}^T}_{s}, \underbrace{{\cal L}_{\alpha}^T, \dots , {\cal L}_{\alpha}^T}_{n-2w-s}\}.
\end{equation}
Or equivalently, we show that the numbers and the sizes of the ${\cal L}$ and ${\cal L}^T$ blocks in \eqref{cka} and \eqref{pg} coincide, i.e., $\beta  + (d - 1)/2 = \alpha$, $s=t$, and $m-2r-t = n - 2w - s$.

For the sizes of the blocks we have
\begin{align} \label{alphasize1}
\beta  +  \frac{d - 1}{2}& =\left\lfloor \frac{rd}{m-2r} \right\rfloor + \frac{d - 1}{2} =  \left\lfloor  \frac{(m-2r)(d-1) + 2rd}{2(m-2r)} \right\rfloor  \\
& = \left\lfloor  \frac{m(d-1)+2r}{2(md - (m(d-1) + 2r))} \right\rfloor = \left\lfloor  \frac{2w}{2(n - 2w)} \right\rfloor = \alpha.
\label{alphasize2}
\end{align}

For the numbers of the blocks, we have
\begin{align*}
t &= rd \mod (m-2r) = ((m-2r)(d-1)/2 + rd) \mod (m-2r) \\
&= (m(d-1)+2r)/2 \mod (md - (m(d-1) + 2r))\\
&=  w \mod (n - 2w) =s
\end{align*}
and
$$m-2r-t = md - m(d-1) - 2r -t= n - 2w - s.$$

Thus ${\cal L}_{W}$ is congruent to the most generic skew-symmetric matrix pencil $\mathcal{W}$ of rank $2w$ obtained in Theorem \ref{th:improvedDeDo}, since they both are skew-symmetric and have the same KCF. Therefore $\orb^c({\cal L}_{W})=\orb^c(\mathcal{W})$.

By Theorem \ref{th:improvedDeDo} for each $m \times m$ skew-symmetric matrix polynomial $P$ of grade $d$ and rank at most $2r$ we have that $\overline{\orb^c}(\mathcal{W})\supseteq \overline{\orb^c}({\cal L}_{P})$, thus $\overline{\orb^c}({\cal L}_{W})\supseteq \overline{\orb^c}({\cal L}_{P})$. Therefore $\overline{\orb^c}({\cal L}_{W}) \cap \GSYL \supseteq \overline{\orb^c}({\cal L}_{P}) \cap \GSYL$, which is equivalent to $\overline{\orb^c({\cal L}_{W}) \cap \GSYL} \supseteq \overline{\orb^c({\cal L}_{P}) \cap \GSYL}$ by Lemma \ref{cl}, and, by definition, is also equivalent to $\overline{\orb}({\cal L}_{W}) \supseteq \overline{\orb}({\cal L}_{P})$, which according to the discussion after \eqref{linorb}, is equivalent to $\overline{\orb}(W) \supseteq \overline{\orb}(P)$. Therefore any $m\times m$ skew-symmetric matrix polynomial of grade $d$ with rank at most $2r$ is in the closed set $\overline{\orb}(W)$.
\end{proof}

For any $P \in \POL_{d, m\times m}^{ss}$, we define the codimension of $\orb (P) \subset \POL_{d, m\times m}^{ss}$ of the set of all the skew-symmetric matrix polynomials with the same eigenstructure as $P$ to be
$
\cod \orb (P) :=
\cod \orb({\cal L}_{P}).$
By \cite[Sect.~6]{Dmyt15}, $\orb({\cal L}_{P})$ is a manifold in the space of skew-symmetric matrix pencils $\PEN_{n \times n}^{ss}$, where $n = md$, and
$
\cod \orb({\cal L}_{P})=
\cod \orb^{c}({\cal L}_{P}),
$
where the codimension of $\orb({\cal L}_{P})$ is considered in the space $\GSYL_{d, m\times m}^{ss}$ and the codimension of $\orb^{c}({\cal L}_{P})$ in $\PEN_{n \times n}^{ss}$.
Therefore, for the generic $m\times m$ skew-symmetric polynomial $W$ with grade $d$ identified in Theorem \ref{mainth}, we deduce from the proof of this theorem that
$\cod \orb(W) = \cod \orb^{c}({\cal L}_{W}) = \cod \orb^{c} ({\cal W})$, where ${\cal W}$ is the pencil in Theorem \ref{th:improvedDeDo} with the identifications $n=md$ and $w = (m(d-1) + 2r)/2$. Therefore, we get from \eqref{codimcomp}
\begin{align*}
\cod \orb(W) & = (md - m(d-1) - 2r - 1) \left(md - m\frac{d-1}{2} - r\right) \\
             & = \frac{1}{2} \, (m - 2r - 1) (m(d+1) - 2r)  \, .
\end{align*}

\section{Conclusions and future work} \label{sect.conclusions}

This paper establishes that {\em there is only one} generic complete eigenstructure for  skew-symmetric matrix polynomials of rank at most $2r$ and odd grade $d$. In order to obtain this result, first, the corresponding result for skew-symmetric pencils is proved, i.e., when $d=1$, and, second, this result is extended to skew-symmetric matrix polynomials of arbitrary odd grades by means of a structure preserving strong-linearization template (also known as a structure preserving companion linearization \cite{DeDM14}) and a delicate translation of the topological properties of the space of skew-symmetric pencils into the space of skew-symmetric matrix polynomials. To the best of our knowledge, this is the first result of this kind obtained for  a class of structured matrix polynomials of fixed (bounded) rank and fixed grade and it is in stark contrast to the results for general (unstructured) matrix polynomials and pencils available in the literature \cite{DeDo08,DeDoLa17,DmDo16}, which establish that there are $rd+1$ generic complete eigenstructures for arbitrary matrix polynomials of rank at most $r$ and grade $d$, instead of only one. This striking reduction in the number of generic eigenstructures is related,  among the other reasons, to the equality of the left and the right minimal indices of skew-symmetric matrix polynomials.

The results in this paper call to natural extensions to other classes of structured matrix polynomials (and pencils), but there are a number of obstacles that make such extensions challenging and nontrivial. The most obvious extension one could have in mind is to skew-symmetric matrix polynomials of rank at most $2r$ and {\em even grade} $d$, but in this case a structured companion linearization is not available in the literature and, in fact, it has been proved that it does not exist for $m\times m$ skew-symmetric matrix polynomials when $m$ is odd \cite[Thm. 7.21]{DeDM14}. Therefore, the techniques used in this paper cannot be used for skew-symmetric matrix polynomials of even grade. For several other structured classes of matrix polynomials of odd grade there are indeed structure preserving companion linearizations \cite{DeDoMac11palin,MMMM10,MMMM11palin} (see also \cite[Sect. 7]{DeDM14} and \cite{MaMT15}, and the references therein), but there are no stratification results available in the literature for these structures. Since in this paper the stratification results for skew-symmetric pencils and polynomials previously developed in \cite{Dmyt15,DmKa14} have played a key role, we see again that the techniques of this paper cannot be directly used to get similar results for other structured matrix polynomials.

\section*{Acknowledgements}

The work of Andrii Dmytryshyn was supported by the Swedish Research Council (VR) under grant E0485301, and by eSSENCE (essenceofescience.se), a strategic collaborative e-Science programme funded by the Swedish Research Council.

The work of Froil\'an M. Dopico was supported by ``Ministerio de Econom\'{i}a, Industria y Competitividad of Spain'' and ``Fondo Europeo de Desarrollo Regional (FEDER) of EU'' through grants MTM-2015-68805-REDT and MTM-2015-65798-P (MINECO/FEDER, UE).

{\small
\bibliographystyle{abbrv}

}
\end{document}